\documentclass[OJMO, manuscript,Unicode]{cedram}

\usepackage[all]{xypic}
\usepackage[dvipsnames]{xcolor}

 \newtheorem{thm}{Theorem}[section]
 
 {\rm}

 \newtheorem{defn}[thm]{Definition}{\rm}
 \newtheorem{assumption}[thm]{Assumption}
 \newtheorem{rem}[thm]{Remark}
 \newtheorem{ex}{Example}
\numberwithin{equation}{section}

\usepackage{comment}

\hypersetup{colorlinks=true, linkcolor=MidnightBlue, citecolor=MidnightBlue, urlcolor=MidnightBlue}

\let\vaccent\v
\def\v{\mathbf{v}} 

\DeclareRobustCommand{\vcaron}[1]{\vaccent{#1}}

\title[Finite convergence of the Moment--SOS hierarchy under hidden convexity]{Finite convergence of the Moment--SOS hierarchy under hidden convexity}
\author{\firstname{Sre\'{c}ko} \lastname{\DJ{}ura\vcaron{s}inovi\'{c}}}
\address{CNRS@CREATE Singapore and College of Computing and Data Science, Nanyang Technological University, Singapore\\
Create Tower $\#$08-01\\1 Create Way\\ Singapore 138602}
\email{srecko001@e.ntu.edu.sg}
\author{\firstname{Jean} \middlename{B.} \lastname{Lasserre}}
\address{LAAS-CNRS and Toulouse School of Economics (TSE)\\
University of Toulouse\\
LAAS, 7 avenue du Colonel Roche\\
31077 Toulouse C\'edex 4, France}
\email{lasserre@laas.fr}
\thanks{This research is supported by the AI Interdisciplinary Institute ANITI  funding through the 
ANITI AI Cluster program under the Grant agreement ANR-23-IACL-0002, and also by the Marie-Sklodovska-Curie European doctoral network TENORS, grant 101120296.
This research is also part of the program DesCartes and is supported by the National Research Foundation, Prime Minister's Office, Singapore under its Campus for Research Excellence and Technological Enterprise (CREATE) program.}
\date{}

\begin{document}
\def\red{\color{red}}
\def\bl{\color{blue}}
\def\ora{\color{orange}}
\def\green{\color{green}}
\def\br{\color{brown}}

\def\la{\langle}
\def\ra{\rangle}
\def\e{{\rm e}}
\def\x{\mathbf{x}}
\def\by{\mathbf{y}}
\def\bz{\mathbf{z}}
\def\F{\mathcal{F}}
\def\R{\mathbb{R}}
\def\T{\mathbf{T}}
\def\N{\mathbb{N}}
\def\K{\mathbf{K}}
\def\bK{\overline{\mathbf{K}}}
\def\Q{\mathbf{Q}}
\def\M{\mathbf{M}}
\def\O{\mathbf{O}}
\def\C{\mathbf{C}}
\def\P{\mathbf{P}}
\def\Z{\mathbb{Z}}
\def\H{\mathbf{H}}
\def\A{\mathbf{A}}
\def\V{\mathbf{V}}
\def\W{\mathbf{W}}
\def\AA{\overline{\mathbf{A}}}
\def\B{\mathbf{B}}
\def\c{\mathbf{C}}
\def\L{\mathbf{L}}
\def\bS{\mathbf{S}}
\def\I{\mathbf{I}}
\def\Y{\mathbf{Y}}
\def\X{\mathbf{X}}
\def\G{\mathbf{G}}
\def\B{\mathbf{B}}
\def\f{\mathbf{f}}
\def\z{\mathbf{z}}
\def\y{\mathbf{y}}
\def\d{\hat{d}}
\def\bx{\mathbf{x}}
\def\y{\mathbf{y}}
\def\h{\mathbf{h}}
\def\g{\mathbf{g}}
\def\v{\mathbf{v}}
\def\tv{\tilde{\mathbf{v}}}
\def\g{\mathbf{g}}
\def\tg{\tilde{\mathbf{g}}}
\def\w{\mathbf{w}}
\def\b{\mathcal{B}}
\def\a{\mathbf{a}}
\def\q{\mathbf{q}}
\def\u{\mathbf{u}}
\def\s{\mathcal{S}}
\def\cc{\mathcal{C}}
\def\co{{\rm co}\,}
\def\cp{{\rm CP}}
\def\tx{\tilde{\x}}
\def\bP{\mathbb{P}}
\def\supmu{{\rm supp}\,\mu}
\def\supnu{{\rm supp}\,\nu}
\def\m{\mathcal{M}}
\def\bR{\mathbf{R}}
\def\om{\mathbf{\Omega}}
\def\c{\mathbf{c}}
\def\s{\mathcal{S}}
\def\k{\mathcal{K}}
\def\la{\langle}
\def\ra{\rangle}
\def\blambda{{\boldsymbol{\lambda}}}
\def\balpha{{\boldsymbol{\alpha}}}
\def\bgamma{{\boldsymbol{\gamma}}}
\def\bbeta{{\boldsymbol{\beta}}}
\def\bphi{{\boldsymbol{\phi}}}
\def\bmu{{\boldsymbol{\mu}}}
\def\bnu{{\boldsymbol{\nu}}}
\def\HM{\mathbf{HM}}
\def\tHM{\widetilde{\mathbf{HM}}}
\def\tbmu{\tilde{\bmu}}

\newcommand{\srecko}[1]{{\color{black}#1}}

\begin{abstract}
We consider polynomial optimization problems
with compact feasible set $\om\subset\R^d$ defined by SOS-concave polynomials
$g_j$ of arbitrary degree, {\color{black} and whose objective function $f$ is not necessarily convex on $\R^d$}.
\srecko{We show that, if $f$ is Hessian-$Q$-module convex over $\om$ in the sense that its Hessian
admits a specific quadratic-module representation, then the standard Moment–SOS hierarchy converges in finitely many steps
without prior knowledge of this hidden (local) convexity. Strong convexity of $f$ on $\om$ is a sufficient condition for the required
Hessian representation. In addition, we give an explicit relaxation order at which exactness occurs.} This demonstrates that a general-purpose hierarchy can adapt to favorable hidden properties of a specific instance without being informed of them, yielding certified global minimizers.
\end{abstract}
\maketitle
\section{Introduction}
\srecko{Given multivariate polynomials $f,g_1,\dots g_m\in\R[\x]$, with $\x=(x_1,\dots,x_d)$, we consider the following polynomial optimization problem (POP)} 
\srecko{
\begin{eqnarray}
 \label{def-pb}
 f^*\,=\,\min_{\x\in\om} \,f(\x), \quad \om=\{\,\x\in\R^d: g_j(\x)\geq0\,,\:j\in\{1,\ldots,m\}\}.
\end{eqnarray}}
\srecko{We assume throughout that $\om$ is compact and that each $-g_j$ is
SOS-convex. Recall that a polynomial $p\in\mathbb{R}[\x]$ is said to be (\textit{globally}) \emph{SOS-convex} if its Hessian is an SOS matrix, that is, if there exists a
polynomial matrix $\mathbf{M}(\x)$ such that $\nabla^2 p(\x)
=
\mathbf{M}(\x)^\top \mathbf{M}(\x)$, for all $\x\in\mathbb{R}^d$. Since SOS-convexity is a tractable sufficient condition for convexity, our
assumptions imply that each constraint-defining polynomial $g_j$ is concave. Consequently,
each superlevel set $\{\x\in\R^d\,:\,g_j(\x)\geq0\}$ is convex, and therefore the feasible
set $\om$ is convex.
\hfill\break\\
The main distinguishing feature of our setting is that the objective
$f$ is assumed to be Hessian-$Q$-module convex over $\om$ (in a sense that will be defined soon), but not necessarily convex on the entire ambient space
$\R^d$, i.e., it may be globally nonconvex. Therefore, the input polynomial data do not necessarily
define a ``convex program'' on $\R^d$. In addition, the explicit certificate of this local convexity property might not be readily available from the problem description only, which is why we use the term \textit{hidden convexity} to describe this scenario.  }
\hfill\break\\
\srecko{The distinction between convexity of $f$ on $\om$ and convexity of $f$ on $\R^d$ matters for global certification.}
Indeed, suppose that Slater's condition holds. Then
the KKT-optimality conditions 
\begin{eqnarray}
\label{KKT}
\nabla f(\x^*)-\sum_{j=1}^m\lambda^*_{j}\nabla g_j(\x^*)\,=\,{\bf 0},\:\lambda^*_j\,g_j(\x^*)=0,\:j\in\{1,\dots,m\}\,,\:\blambda^*\geq{\bf 0},
\end{eqnarray}
must hold at some local minimizer $\x^*\in\om$ (for some $\blambda^*\geq{\bf 0}$).
\\
\srecko{If $f$ is convex on the entire space $\R^d$, then we obtain that}, 
\begin{eqnarray}
\label{Lagrangian}
\mathcal{L}(\x,\blambda^*)-f^*\,:=\,f(\x)-f^*-\sum_{j=1}^m\lambda^*_j\, g_j(\x)\,\geq\,0,\quad\forall\x\in\R^d.
\end{eqnarray}
More precisely, because $\mathcal{L}(\cdot,\blambda^*)$  is convex and 
$\nabla_{\x}\mathcal{L}(\x^*,\blambda^*)=0$, the point $\x^*$ is also a global minimizer of the 
Lagrangian $\mathcal{L}(\cdot,\blambda^*)$ on the whole $\R^d$.\\
\srecko{If $f$ is only convex on $\om$, this conclusion generally fails, i.e., $\mathcal{L}(\cdot,\blambda^*)$ 
is not necessarily convex on $\R^d$ (typically when $\displaystyle\mathrm{deg}(f)>\max_{j\in\{ 1,\dots,m\}}\{\mathrm{deg}(g_j)\}$). In other words, the KKT conditions still characterize the
constrained minimizer, but they do not imply that
$\x^\star$ minimizes $\mathcal L(\cdot,\blambda^\star)$ over the whole space.} 
For a simple univariate illustration, consider the following POP:  $f^*=\min\{\,x^3:\:x(1-x)\geq0\}$
with $(x^*,\lambda^*)=(0,0)$.
\hfill\break\\
\srecko{A local optimization method may return a feasible point
$\x^\star\in\om$ satisfying the KKT conditions. If $\om$ is convex and $f$ is
known to be convex on $\om$, then every local minimizer is globally optimal.
However, when the convexity of $f$ on $\om$ is hidden, the output of a local
method alone \textit{does not} certify global optimality. Moreover, even deciding
whether a polynomial is globally convex is NP-hard \cite{Ahmadi}, so verifying
such convexity explicitly may itself be computationally intractable. This
motivates methods that can exploit hidden convexity, i.e., the one that is not
apparent from the polynomial description of the problem.}
\hfill\break\\
\srecko{In this paper, we analyze  one such method, namely the standard hierarchy of Moment--SOS relaxations associated with
\eqref{def-pb}. At each relaxation order $n$, the original polynomial
optimization problem is replaced by a semidefinite relaxation which returns a lower bound $\rho_n\leq f^\star$. Under
the Archimedean assumption, the sequence $(\rho_n)_n$ is monotone
nondecreasing and converges asymptotically to $f^\star$ \cite{lass-2001}. The precise
semidefinite formulation is given in Section \ref{subsec: mom}.
\\
Nie showed that finite convergence of the sequence $(\rho_n)_n$ holds generically, without requiring convexity of either the objective or
the constraints \cite{Nie}. For a given POP instance, however, checking
the relevant genericity conditions may be computationally prohibitive.
More generally, deciding whether the Moment--SOS hierarchy converges finitely
is itself a hard problem, as shown by Vargas \cite{vargas}.}
\hfill\break\\
Of course, when the information that  
$f$ and $-g_j$ are convex is known \`a priori, 
some algorithms can be more efficient than the Moment--SOS hierarchy, but the interesting observation is 
when the convexity of $f$ is \emph{hidden}. 
\\
Remarkably, the Moment--SOS hierarchy blindly \textit{recognizes} and exploits convexity properties of certain POP instances, without any \emph{\`a priori} knowledge of these properties, thereby ensuring finite convergence of the resulting sequence of lower bounds.
 It is a nice feature of a general purpose method to be able to adapt its behavior by benefiting from strong properties of some specific instances, without being ``informed'' of such properties.
\subsection*{Related work}
The behavior of the Moment--SOS hierarchy is well understood for several
classes in which convexity is global. If $f$ and $-g_j$ are SOS-convex on
$\R^d$, a Jensen inequality for positive Riesz functionals yields exactness
at the first semidefinite relaxation of the hierarchy \cite{lass-conv}. 
\\
For
convex polynomial programs that are not necessarily SOS-convex, finite
convergence follows under additional nondegeneracy assumptions,  but without
knowing at which step of the hierarchy convergence takes place \cite{deklerk-laurent}. More precisely, \cite[Corollary 3.3]{deklerk-laurent} establishes the finite convergence  of the Moment--SOS hierarchy
for convex programs  
when (i) Slater's condition holds, (ii) the quadratic module associated with the $g_j$'s is Archimedean, and (iii) $\nabla^2f(\x^*)\succ0$ at every minimizer $\x^*$. 
Also, if
$\om\subset\B(0,R)$, the same result holds under conditions (ii) and (iii), but now satisfied by
the Lagrangian $\mathcal{L}$ rather than by $f$. Two key ingredients in \cite{deklerk-laurent} 
are a result of Scheiderer \cite{scheiderer} and the (convex) Farkas Lemma 
\cite[Prop. 5.3.1]{bertsekas}. 
\\
Finally, it is shown in \cite{lass-rep} 
that 
if both $f$ and 
$\om$ are convex (without the $g_j$'s being necessarily concave), all points that satisfy 
the KKT-optimality conditions are global minimizers. 
\hfill\break\\
\srecko{Prior works, therefore, do not cover the scenario when the
objective is convex only on $\om$ and may be nonconvex away from the feasible
set, which is exactly the regime considered here.}
\subsection*{Contribution}
\srecko{Let $g_0:\x\to 1$ be the constant polynomial. We start by introducing the notion of \emph{Hessian-$Q$-module convexity over $\om\subset\R^d$}. Namely, we say that a polynomial 
$f\in\R[\x]$ is Hessian-$Q$-module convex over $\om$ if}
\begin{eqnarray}
\label{intro-matrix-put}
\nabla^2f(\x)\,=\,\sum_{j=0}^m\L_j(\x)\L_j(\x)^\top\,g_j(\x),\quad\forall\x\in\R^d,\end{eqnarray}
for some integers $r_{j}\in\N$
and  matrix polynomials $\L_j\in\R[\x]^{d\times r_{j}}$, $j\in\{0,\ldots,m\}$.\\
This definition is a natural extension of global SOS-convexity on $\R^d$ (introduced in
\cite{helton}, see also \cite{lass-conv})
to local SOS-convexity on $\om$. \srecko{It is also closely related to the notion of $Q$-module convexity introduced
in \cite{Nie-2012} for studying semidefinite representability of some convex sets (not for optimization). The condition in \eqref{intro-matrix-put} is a
second-order positivity certificate expressed through the Hessian, whereas
$Q$-module convexity is formulated through an SOS representation of the
first-order remainder
$R_f: (\x,\y)\to f(\x)-f(\y)-\nabla f(\y)^\top(\x-\y)$.
A detailed comparison between these two notions is provided in
Section \ref{sec: discussion}.}
\hfill\break\\
Notice that \eqref{intro-matrix-put} can be viewed as a Putinar-type \emph{algebraic certificate} that 
$f$ is  convex on $\om$.
In particular, if $f$ is strongly convex on $\om$, then by
the Putinar's matrix Positivstellensatz \cite[Theorem 2.19]{lass-CUP}, \eqref{intro-matrix-put} holds and so $f$ is Hessian-Q-module convex over $\om$. In principle,
checking whether \eqref{intro-matrix-put} holds can be done efficiently by solving a semidefinite program of finite size determined by an \`a priori bound on the maximum degree
appearing in entries of matrix polynomials $\L_j(\x)$.
\hfill\break\\
We emphasize that the POPs considered here define convex optimization problems, but not convex programs in the standard sense, since $f$ is not convex on the whole $\mathbb{R}^d$.
\hfill\break\\
\srecko{Our main result states that when $f$ is Hessian-$Q$-module convex over $\om$ (i.e., \eqref{intro-matrix-put} holds), the  Moment--SOS hierarchy converges in finitely many steps. In addition,  an explicit lower bound on the order $n$ of an exact relaxation is given by
\begin{eqnarray}
   \displaystyle 1+\max\left\{\max_{j\in\{1,\dots,m+1\}}\{a_j+d_j\},a_0\right\},\,\text{with}\, a_j:=\max_{(k,l)\in\{1,\dots,d\}\times\{1,\dots,r_j\}}\deg(\L_j{(k,l)}),
\end{eqnarray}
where $\L_j{(k,l)}$ denotes the polynomial entry $(k,l)$ of the matrix $\L_j$ from \eqref{intro-matrix-put}, and $\displaystyle d_j=\left\lceil\frac{\deg g_j}{2}\right\rceil$.}
\hfill\break\\
\srecko{\emph{A key feature of our result is that finite convergence does not require
\emph{à priori} verification of the underlying convexity structure: one need
not check that $f$ is strongly convex on $\Omega$, nor that
\eqref{intro-matrix-put} holds. The relevant convexity may therefore remain
\emph{hidden}. Nevertheless, the standard, unmodified Moment--SOS hierarchy
\eqref{intro-relax} still converges finitely. Our result thus identifies a
distinct mechanism for finite convergence, arising from convexity certified
only on the feasible set, even though this certificate is neither known nor
used in constructing the hierarchy.}}
\hfill\break\\
Crucial for our proofs and of independent interest, we provide a Jensen-type inequality (see Theorem \ref{lem1}) for
polynomials that are \srecko{Hessian-$Q$-module convex over $\om$}, and linear functionals $\phi$ whose associated  localizing matrices $\M_{n-d_j}(g_j\cdot\bphi)$, $j\in\{0,\ldots,m\}$, are positive semidefinite (thus an extension beyond probability measures on $\om$, but for a restricted class of convex functions).
\hfill\break\\
We emphasize once again that, \emph{if one knows that $f$ is convex on 
$\om$}, then running a local algorithm that converges to a KKT point $\x^*\in\om$ might be enough, because 
 \eqref{KKT} provides $\x^*$ with a certificate of global optimality on $\om$, and there is no need to run
the Moment--SOS hierarchy. But if one is unable to check beforehand whether 
$f $ is convex on $\om$ (an NP-hard problem), then the local KKT-optimality conditions are not sufficient 
to ensure that $\x^*$ is a global minimizer. %

\section{Preliminaries}

\subsection{Notation and main definitions}
Let $\R[\x]=\R[x_1,\ldots,x_d]$ be the ring of $d$-variate polynomials in the variables
$\x=(x_1,\ldots,x_d)$, and let $\R[\x]_{n}\subset\R[\x]$ be the vector space of polynomials of total degree at most $n$. Let $\Sigma[\x]$ be the cone of sum of squares (SOS) polynomials and let $\Sigma[\x]_n := \Sigma[\x] \cap
\R[\x]_{2n}$.
For $\balpha\in\N^d$, let
$\displaystyle\vert\balpha\vert:=\sum_{i=1}^d\alpha_i$, and denote by  $\N^d_n=\{\balpha\in\N^d: \vert\balpha\vert\leq n\}$ the finite set of multi-indices of size $s(n):=\binom{d+n}{n}$. Let $\v_n(\x)=(\x^{\balpha})_{\balpha\in\N^d_n}$
be the vector of all monomials with total degree bounded by $n$. 
For a real symmetric matrix $\A$,
the notation $\A\succeq{\bf0}$ (resp. $\A\succ{\bf0}$) means that $\A$ is positive semidefinite
(resp. positive definite). 
A polynomial $f\in\mathbb{R}[\x]_n$ is written $\displaystyle \x\mapsto f(\x)\,=\,\sum_{\alpha\in\N^d}f_{\balpha}\,\x^{\balpha}\,=\,\la\f,\v_n(\x)\ra\,,$ with vector of coefficients $\displaystyle \f=(f_{\balpha})_{\balpha\in\N^d_n}\in\R^{s(n)}$ in the basis $\v_n(\x)$. \srecko{The} Euclidean ball centered at origin and having radius $R>0$ is denoted by $\B({\bf 0},R)$.
\subsection*{Riesz linear functional.}
Given a sequence $\bphi=(\phi_{\balpha})_{\balpha\in\N^d}$ (in bold), its associated Riesz functional is the linear mapping
$\phi:\mathbb{R}[\x]\to\mathbb{R}$ (not in bold) defined by:
\begin{eqnarray}
\label{Riesz}
f\:(=\sum_{\balpha} f_{\balpha}\,\x^{\balpha})\quad \mapsto \phi(f)\,=\,\sum_{\balpha\in\N^d}f_{\balpha}\,\phi_{\balpha}\,=\,\langle\f,\bphi\rangle.
\end{eqnarray}
A sequence $\bphi$ has a \emph{representing} measure if its associated Riesz linear functional $\phi$ is a (positive) Borel measure on $\R^d$, in which case,
\begin{eqnarray}
\phi_{\balpha}\,=\,\int_{\R^d}\x^{\balpha}\,d\phi,\quad\forall \balpha\in\N^d.
\end{eqnarray}
Given a sequence $\bphi=(\phi_{\balpha})_{\balpha\in\N^d}$ and a polynomial $\R[\x]\ni g:\x\mapsto g(\x)=\sum_{\bgamma}g_{\bgamma}\,\x^{\bgamma}$, we define the new sequence
$g\cdot\bphi$ via
\begin{eqnarray}
(g\cdot\bphi)_{\balpha}\,:=\,\phi(\x^{\balpha}\,g)\,=\,\sum_{\bgamma\in\N^d}g_{\bgamma}\,\phi_{\balpha+\bgamma}\,,\quad\forall\balpha\in\N^d,
\end{eqnarray}
so that its associated Riesz linear functional, denoted by  $g\cdot\phi$, satisfies
\begin{eqnarray}
    \srecko{(}g\cdot\phi\srecko{)}(f)\,=\,\phi(g\,f),\quad \forall f\in\R[\x].
\end{eqnarray}
\subsection*{Moment matrix.}
The {\it moment} matrix of order $n\in\N$ associated with a sequence
$\bphi=(\phi_{\balpha})_{\balpha\in\N^d}$ (or, equivalently, with the Riesz linear functional $\phi$), 
is the real symmetric matrix of size $s(n)\times s(n)$ denoted by $\M_n(\bphi)$ (or $\M_n(\phi)$)
with rows and columns indexed by $\N^d_n$, and whose entry
$(\balpha,\bbeta)$ corresponds to $\phi(\x^{\balpha+\bbeta})=\phi_{\balpha+\bbeta}$, for every $\balpha,\bbeta\in\N^d_n$.
Therefore, $\M_n(\bphi)$ depends only on moments $\phi_{\balpha}$ of degree at most $2n$.
Moreover, if $\bphi$ has a representing measure $\phi$, then
$\M_n(\bphi)\succeq{\bf0}$ because $\langle\f,\M_n(\bphi)\,\f\rangle=\int f^2d\phi\geq0$, for all $f\in\mathbb{R}[\x]_n$.


\subsection*{Localizing matrix.}
With $\bphi$ as above and $g\in\mathbb{R}[\x]$ (with $g(\x)=\sum_\gamma g_\gamma\x^\gamma$), the \textit{localizing} matrix associated with $\bphi$
and $g$ is the moment matrix $\M_n(g\cdot\bphi)$ associated with the sequence $g\cdot\bphi$. That is,
$\M_n(g\cdot\bphi)$ is the real symmetric matrix 
with rows and columns indexed by $\N^d_n$, and whose entry $(\balpha,\bbeta)$ is just 
$(g\cdot\bphi)_{\balpha+\bbeta}=\phi(\x^{\balpha+\bbeta}g)$, that is,
$\M_n(g\cdot\bphi)(\balpha,\bbeta)=\sum_{\gamma}g_{\bgamma} \phi_{\balpha+\bbeta+\bgamma}$, for every $\balpha,\bbeta\in\N^d_n$.
If $\bphi$ has a representing measure $\phi$ whose support is contained in the set $\{\x\in\R^d:g(\x)\geq0\}$ then $\M_n(g\cdot\bphi)\succeq{\bf0}$ for all $n$, because 
\begin{eqnarray} \langle\f,\M_n(g\cdot\bphi)\,\f\rangle\,=\,\srecko{(}g\cdot\phi\srecko{)}(f^2)\,=\,\phi(f^2g)\,=\,\int f^2\,g\,d\phi\,\geq\,0,\quad\forall f\,\in\mathbb{R}[\x]_n.
\end{eqnarray}
\subsection*{The Moment--SOS hierarchy}\label{subsec: mom}
\srecko{Let $d_j=\lceil\mathrm{deg}(g_j)/2\rceil$, $j\in\{0,\ldots,m\}$, $d_f=\lceil\mathrm{deg}(f)/2\rceil$ and let $n\geq n_{\min}:=\max[d_f,\max_jd_j]$,.} The \srecko{order}-$n$ semidefinite relaxation of Moment--SOS hierarchy to solve the POP \eqref{def-pb} corresponds to:
\begin{eqnarray}
    \label{intro-relax}
\rho_ n=\min_{{\bphi}\in\R^{\N^d_{2n}}}\,\{\,\phi(f):\,\phi(1)=1,\:\M_{n-d_{j}}(g_j\cdot\bphi)\succeq {\bf 0},\:j\in\{0,\ldots,m\}\}.\end{eqnarray}
Then, $\rho_n\leq f^*$ for 
all $n\geq\max[d_f,\max_jd_j]$, and if the quadratic module defined by the $g_j$'s is Archimedean (see discussion around \eqref{eq:truncquadmod}),  $\rho_n\uparrow f^*$ as $n\to+\infty$. \srecko{We say that the hierarchy has \emph{finite convergence}, or that
the order-$n$ relaxation is \emph{exact}, if $\rho_n=f^*$ for some finite
$n$. A standard sufficient à posteriori certificate of exactness is the
flat truncation condition \cite{CurtoFialkow}. More precisely, if an optimal solution
$\bphi^{(n)}$ of \eqref{intro-relax} satisfies, for some
$t\in\{n_{\min},\ldots,n\}$,
\begin{eqnarray}\label{eq:flat-truncation}
\operatorname{rank}\M_t(\bphi^{(n)})
=\operatorname{rank}\M_{t-n_{\min}}
(\bphi^{(n)}),
\end{eqnarray}
then the corresponding truncated sequence admits a $t$-atomic
representing measure supported on $\om$ whose atoms can be extracted and are
global minimizers of $f$ on $\om$.} For more details on the Moment--SOS hierarchy (its generic finite convergence and other properties), the interested reader is referred to \cite{lass-CUP, acta, Nie}.

\subsection*{SOS-convex polynomials.}
\begin{defn}
\label{def-0} A polynomial $p\in\R[\x]$ is \emph{(globally) SOS-convex} if its Hessian $\nabla^2p$ is a matrix-SOS.
That is, there exists a matrix polynomial $\L\in\R[\x]^{d\times r}$ such that
$\nabla^2 p(\x)=\L(\x)\,\L(\x)^T$, for all $\x\in\R^d$. 
A polynomial $p$ is SOS-concave if $-p$ is SOS-convex. 
\end{defn}
\begin{defn}
\label{def-1}
With $\om\subset\R^d$ as in \eqref{def-pb} (hence convex), a  polynomial $f\in\R[\x]$ is \srecko{\emph{Hessian-$Q$-module convex over $\om$}} if its Hessian $\nabla^2f$ satisfies
\begin{eqnarray}
\label{def-sos-om}
\nabla^2 f(\x)\,=\,\sum_{j=0}^m\L_{j}(\x)\,\L_{j}(\x)^\top\,g_{j}(\x),
\quad\forall \x\in\R^d,
\end{eqnarray}
for some matrix polynomials $\L_{j}\in\R[\x]^{d\times r_j}$, with $r_j\in\N$ for $j\in\{0,\dots,m\}$. 
\end{defn}

\subsection*{Quadratic module}
\begin{defn}\label{def:quadmod}
   Let
$\g =(g_0,g_1,\ldots,g_m)\in\R[\x]^{m+1}$. The set $Q(\g)\subset\R[\x]$ defined by
\begin{eqnarray}
\label{quad-module}
Q(\g)\,:=\left\{\sum_{j=0}^{m}\sigma_{j}\,g_{j}:\:\sigma_{j}\in\Sigma[\x],\,j\in\{0,\dots,m\}\right\}.
\end{eqnarray}
is called a quadratic module  generated by $g_1,\ldots,g_m$.  Its truncated version of order $n\in\N$, namely
$Q_{n}(\g)\subset\R[\x]_{n}$, reads
\begin{eqnarray}\label{eq:truncquadmod}
    Q_{n}(\g)\,:=\left\{\sum_{j=0}^{m}\sigma_{j}\,g_{j}:\:\sigma_{j}\in\Sigma[\x],\:\mathrm{deg}(\sigma_jg_j)\,\leq n,\:j\in\{0,\ldots,m \}\right\}.
\end{eqnarray}
\end{defn}
{\color{black}The quadratic module $Q(\g)$ is said to be Archimedean if there exists $R>0$ such that $R^2-||\x||^2\in Q(\g)$ (therefore, $\om$ is necessarily compact).}
\\
If $f$ is strongly convex on $\om$ \footnote{A polynomial $f$ is strongly convex on $\om$ if and only if there exists $\mu>0$ such that $\nabla^2f
(\x) - \mu {\bf I} \succeq {\bf 0}$ for all $\x\in\om$.} and the quadratic module $Q(\g)$ is Archimedean, then by Putinar's matrix Positivstellensatz (\srecko{see \cite[Corollary 1]{SchererHol} or  \cite[Theorem 2.22]{lass-CUP}}), equation \eqref{def-sos-om} holds, that is, $f$ is Hessian-$Q$-module convex over $\om$.

\section{Main result}
Since $\om$ is compact, we may and will assume that
$\om\subset \B({\bf0},1)$ possibly after re-scaling. Hence, we can include in the definition of 
$\om$ the redundant quadratic constraint $g_{m+1}(\x)\geq0$, with $g_{m+1}:\x\mapsto 1-\Vert\x\Vert^2$ (which is SOS-concave).
This ensures that the quadratic module $Q(\g)$ in \eqref{quad-module} is Archimedean.
Therefore, throughout the paper, we assume the following:
{%
\renewcommand{\theenumi}{(\roman{enumi})}%
\renewcommand{\labelenumi}{\theenumi}%
\begin{assumption}\label{ass-1}
\begin{enumerate}
\item The set $\om\subset\R^d$ in \eqref{def-pb} is compact, and $g_{m+1}$ is one of the defining constraints of $\om$;
\item  For every $j\in\{1,\ldots,m+1\}$, the constraint $g_j$ is SOS-concave  (hence $\om$ is convex);
\item  Objective function $f$ is \srecko{Hessian-$Q$-module convex over $\om$} (i.e., it satisfies \eqref{def-sos-om}), but $f$ is \emph{not} necessarily convex on $\R^d$. \srecko{Recall that, in this case, if  $\L_j{(k,l)}$ denotes the polynomial entry $(k,l)$ of the matrix $\L_j$ from \eqref{def-sos-om}, then we set $\displaystyle a_j:=\max_{(k,l)\in\{1,\dots,d\}\times\{1,\dots,r_j\}}\deg(\L_j{(k,l)})$.}
\end{enumerate}
\end{assumption}
}

\subsection{A Jensen inequality for linear functionals}
Before stating Theorem \ref{lem1}, the main result underpinning our analysis of the finite convergence of the Moment--SOS hierarchy under Assumption \ref{ass-1}, we make the following remark.
\srecko{\begin{rem}[Feasibility of the first-order pseudo-moment vector]\label{rem1}
    Let
$\bphi$ be feasible for the order-$n$ moment relaxation \eqref{intro-relax} and define the first-order pseudo-moment vector $\phi({\x}):=(\phi(x_1),\dots,\phi(x_d)).$ For each $j\in\{0,\dots,m+1\}$, the localizing constraint ${\bf M}_{n-d_j}(g_j\cdot\bphi)\succeq {\bf 0}$
implies $\phi(g_j)\geq0.$
Moreover, since $g_j$ is SOS-concave, the Jensen inequality for SOS-concave
polynomials \cite[Theorem 5.13]{lass-CUP} gives $0\leq\phi(g_j)\leq g_j(\phi(\x)).$
Consequently,  $\phi(\x)\in\om$.
\end{rem}}

{\color{black}
\begin{thm}
\label{lem1}
Let Assumption \ref{ass-1} hold, 
and  let $\bphi=(\phi_\balpha)_{\balpha\in\N^d_{2n}}$ be such that
\begin{eqnarray}
\label{on-phi}
\phi(1)=1\,\text{and}\,\:\M_{n-d_j}(g_j\cdot\bphi)\,\succeq {\bf 0},\quad j\in\{0,\ldots,m+1\}.\end{eqnarray} 
Then, as soon as $\displaystyle n\geq 1+\max\left\{\max_{j\in\{1,\dots,m+1\}}\{a_j+d_j\},a_0\right\}$, we have 
\begin{eqnarray}\label{Jensen}
    \phi(f)\,\geq\,f(\phi(\x)), \quad\text{where}\quad \phi(\x):=(\phi(x_1),\ldots,\phi(x_d)).
\end{eqnarray}
\end{thm}}
{\color{black}
\begin{proof}
Let $\z\in\om$ be fixed, arbitrary. As $\phi$ is linear and satisfies $\phi(1)=1$, we get $\phi(f-f(\z))=\phi(f)-f(\z)$. We will prove that
 \begin{eqnarray}
  \phi(f)-f(\z)=\phi(f-f(\z))\,\geq\,\phi(\la\nabla f(\z),\x-\z\ra)\:=\,\la\nabla f(\z),\phi(\x)-\z\ra,   
 \end{eqnarray}
 and therefore, using Remark \ref{rem1} substituting $\z$ with $\phi(\x)\in\om$, one obtains the desired result \eqref{Jensen}.
 \\
 We start by writing the second-order Taylor formula with integral remainder, which has previously been
used in the SOS-convexity literature  (see, for example,
\cite{helton,Ahmadi2}), namely: \begin{eqnarray}\label{eq:integralTaylor}
 f(\x)-f(\z)=
 \la\nabla f(\z),\x-\z\ra
 +\Bigg\langle(\x-\z),\left(\int_0^1(1-t)\,\nabla^2f(\z+t(\x-\z))\,dt\right)\,(\x-\z)\Bigg\rangle.\end{eqnarray}
 As the maximum degree (in $t$) of the entries of the integrand
 $(1-t)\nabla^2f(\z+t(\x-\z))$ is $2d_f-1$, the entire integral from \eqref{eq:integralTaylor} may be computed using a cubature formula for Lebesgue measure on $[0,1]$, supported on $s$ points $t_1,\ldots t_s\in [0,1]$, with weights $\bgamma>{\bf0}$. The formula being exact for polynomials 
 of degree $2d_f-1$, we obtain
  \begin{eqnarray}
  f(\x)-f(\z)=\la\nabla f(\z),\x-\z\ra+\sum_{k=1}^s\gamma_{k}
 \Bigg\langle(\x-\z),\left((1-t_k)\,\nabla^2f(\z+t_k\,(\x-\z))\right)\,(\x-\z)\Bigg\rangle.\end{eqnarray}
 Next, since for $\x,\z\in\om$ and each $k\in\{1,\ldots,s\}$, the point $\z+t_k(\x-\z)$ belongs to $\om$, Assumption \ref{ass-1} yields
  \begin{eqnarray}
 \nabla^2f(\z+t_k\,(\x-\z))=\sum_{j=0}^{m+1}\L_j(\z+t_k(\x-\z))\L_j(\z+t_k(\x-\z))^\top g_j(\z+t_k(\x-\z)).
 \label{aux-100}
 \end{eqnarray}
 Substituting \eqref{aux-100} into \eqref{eq:integralTaylor}, we get 
 \begin{eqnarray}
 \label{taylor-sum}
f(\x)-f(\z)-\nabla f(\z)^\top(\x-\z)
=\sum_{k=1}^s\gamma_k(1-t_k)
\sum_{j=0}^{m+1}\sigma_{kj}(\x,\z)
g_j(\z+t_k(\x-\z)),    
 \end{eqnarray}
where $\sigma_{kj}(\x,\z):=
\left\|\L_j(\z+t_k(\x-\z))^\top(\x-\z)\right\|^2$. Notice that, for fixed $\z$, $\sigma_{kj}(\cdot,\z)$ is SOS and belongs to
$\Sigma[\x]_{a_j+1}$. 
\\
For $j\geq1$, let us now define the line-segment remainder
\begin{eqnarray}\label{q-kj}
q_{kj}(\x,\z):=
g_j((1-t_k)\z+t_k\x)
-t_kg_j(\x)-(1-t_k)g_j(\z).
\end{eqnarray}
Because $-g_j$ is SOS-convex, the equivalent Jensen characterization in
\cite[Theorem~3.1]{Ahmadi2}, applied with weight $t_k$, gives
$q_{kj}\in\Sigma[\x,\z]_{d_j}$. We therefore obtain
\begin{eqnarray}
 \label{taylor-sum2}
f(\x)-f(\z)-\nabla f(\z)^\top(\x-\z)
=\sum_{k=1}^s\gamma_k(1-t_k)
\sum_{j=0}^{m+1}\sigma_{kj}(\x,\z)
(q_{kj}(\x,\z)+t_kg_j(\x)+(1-t_k)g_j(\z)).  
 \end{eqnarray}
 Recall that, by hypothesis, $\bphi$ is feasible for the relaxation of order $n$, with $\displaystyle n\geq 1+\max\left\{\max_{j\in\{1,\dots,m+1\}}\{a_j+d_j\},a_0\right\}$. We now apply $\phi$ in the $\x$ variables, with
$\z$ fixed, to deduce that every term on the right-hand side of \eqref{taylor-sum2} is nonnegative:
\begin{enumerate}
\item For $j\geq1$, our assumption gives
$a_j+1\leq n-d_j$. Thus, the localizing constraint implies
$\phi(\sigma_{kj}g_j)\geq0$. For $j=0$, the same conclusion follows from
$n\geq a_0+1$ and $\M_n(\bphi)\succeq{\bf 0}$.
\item By Remark \ref{rem1}, when
$\z=\phi(\x)$ one has $g_j(\z)\geq0$. Moreover,
$\phi(\sigma_{kj})\geq0$ because $n\geq a_j+1$.
\item The product of two SOS polynomials is SOS. Hence,
$\sigma_{kj}q_{kj}\in
\Sigma[\x]_{a_j+d_j+1}$ after $\z$ is fixed, and consequently
$\phi(\sigma_{kj}q_{kj})\geq0$. 
\end{enumerate}
Substituting $\z=\phi(\x)$, we get $\phi(\nabla f(\z)^\top(\x-\z))
=\nabla f(\z)^\top(\phi(\x)-\z)=0$, and since all remaining terms on the right-hand side of  \eqref{taylor-sum2} are nonnegative by the three
points above, we obtain $\phi(f)-f(\phi(\x))\geq0$, which is the desired result.
 \end{proof}}
   
   \subsection{Finite convergence of the Moment--SOS hierarchy}
Recall that, under Assumption \ref{ass-1},
 the quadratic module \eqref{quad-module} associated with $\g$ is Archimedean, and that the $n$-th moment semidefinite relaxation of POP \eqref{def-pb} has the following form:
\begin{eqnarray}
 \label{relax}
  \rho_n =\displaystyle\min_{{\bphi}\in\R^{\N^d_{2n}}}\,\{\,\phi(f):\:\phi(1)=1,\:\M_{n-d_j}(g_j\cdot\bphi)\,\succeq {\bf 0},\, j\in\{0,\ldots,m+1\}\}.
 \end{eqnarray}
 {\color{black}
\begin{thm}
\label{th1}
 Let Assumption \ref{ass-1} hold.  
 Then, for every $\displaystyle n\geq n_{\min}:=\max\left\{\max_{j\in\{0,\ldots,m+1\}}d_j,d_f\}\right\}$, the semidefinite relaxation \eqref{relax} has an optimal solution $\bphi^{(n)}$. In addition, if $\displaystyle n\geq 1+\max\left\{\max_{j\in\{1,\dots,m+1\}}\{a_j+d_j\},a_0\right\}$, we have $\rho_n=f^*=f(\phi^{(n)}(\x))$, i.e., the point $\phi^{(n)}(\x)=(\phi^n(x_1),\dots,\phi^n(x_d))\in\om$ is globally optimal in \eqref{def-pb}.
  \end{thm}
 \begin{proof}
 Recall that $g_{m+1}(\x)=1-\Vert\x\Vert^2$ and that $\M_{n-d_{m+1}}(g_{m+1}\cdot\bphi)\succeq {\bf 0}$. This implies that $\phi(x_i^{2n})\leq 1$ for every  $i\in\{1,\ldots,d\}$. Hence, by \cite[Proposition  2.38]{lass-CUP},
 \begin{eqnarray}
\vert\phi_\balpha\vert\,\leq\,\max\left\{1,\max_{i\in\{1,\dots,d\}}\phi(x_i^{2n})\right\}\,\leq\, 1,\quad \forall \balpha\in\N^d_{2n}.
 \end{eqnarray}
Therefore, the feasible set of \eqref{relax} is compact, implying that \eqref{relax} has an optimal solution $\bphi^{(n)}\in\R^{\N^d_{2n}}$.
By construction, $\rho_n\leq f^*$ for all $n\geq n_{\min}$. 
\hfill\break
Next let $\bphi^{(n)}$ be an optimal solution, hence
 satisfying  $\phi^{(n)}(f)=\rho_n\leq f^*$. By Remark \ref{rem1}, 
 $\bphi^{(n)}\in\om$, which in turn implies $f^*\leq f(\phi^{(n)}(\x))$. \\
 Finally, for $\displaystyle n\geq 1+\max\left\{\max_{j\in\{1,\dots,m+1\}}\{a_j+d_j\},a_0\right\}$, Theorem \ref{lem1} guarantees that
 \begin{eqnarray}
     f^*\,\leq\,f(\phi^{(n)}(\x))\,\leq \phi^{(n)}(f)\,=\,\rho_{n}\leq\,f^*,
 \end{eqnarray}
 which yields the desired result.
 \end{proof}}
\noindent
{\color{black}
When solving \eqref{relax} at step $n\geq n_{\min}$ of  the Moment--SOS hierarchy, 
it is always recommended to check whether 
$f(\phi^{(n)}(\x))=\rho_n$ at an optimal solution $\bphi^{(n)}$. 
If Assumption \ref{ass-1} holds, then by Theorem \ref{th1}, this equality is guaranteed to occur for some finite $n$ large enough. We emphasize that this finite convergence result holds for the \textit{standard} Moment--SOS hierarchy whose description does not involve any information on this hidden convexity.}
\hfill\break\\
Finally, we give a toy example that illustrates the main notions and concepts introduced above. 
\begin{ex}[Illustration when $d=2$] Consider the following optimization problem
\begin{eqnarray}
 \label{illustration}
 f^*\,=\,\min_{\x\in\om:=\{\x\in\R^2\, :\, \underbrace{1-x_1^2-x_2^2}_{:=g_1(\x)}\geq 0\}} \,f(\x):=\frac16 x_1^3+\frac12 x_1^2+x_1+\frac12 x_2^2.
\end{eqnarray}
Observe that the quadratic constraint $g_1$ is SOS-concave, therefore, $\om$ is convex, but also compact. In addition, $Q(\g)$ is Archimedean by construction. A simple calculation shows that 
\begin{eqnarray}
    \nabla^2 f(\x)=
\begin{bmatrix}
x_1+1 & 0\\
0 & 1
\end{bmatrix}, \quad\forall\x\in\R^2.
\end{eqnarray}
The function $f$ is clearly not convex on $\R^2$. 
Next, define the SOS polynomials $\sigma_0:\x\mapsto\frac12(x_1+1)^2+\frac12 x_2^2 \in \Sigma[\x]_1$ and $ \sigma_1:\x\mapsto\frac12\in\Sigma[\x]_0.$
Then the following identity holds \emph{exactly}:
\begin{eqnarray}
    x_1+1
= \sigma_0(\x)+g_1(\x)\sigma_1(\x),\, \forall\x\in\R^2.
\end{eqnarray}
Consequently,
\begin{eqnarray}
    \nabla^2 f(\x)
&&=
\bS_0(\x)\,\underbrace{g_0(\x)}_{:=1}\;+\;\bS_1(\x)\,g_1(\x)=\begin{bmatrix}\sigma_0(\x)&0\\0&1\end{bmatrix}+\begin{bmatrix}\sigma_1(\x)&0\\0&0\end{bmatrix}g_1(\x)\\
&&= \L_0\L_0^\top+\L_1\L_1^\top g_1(\x), \quad\text{with}\quad\L_0(\x):=
\begin{bmatrix}
\frac{1+x_1}{\sqrt2} & \frac{x_2}{\sqrt2} & 0\\[1mm]
0 & 0 & 1
\end{bmatrix},
\,
\L_1(\x):=
\begin{bmatrix}
\frac{1}{\sqrt2}\\[1mm]
0
\end{bmatrix},
\end{eqnarray}
which certifies that $f$ is \srecko{Hessian-$Q$-module convex over $\om$}.
Finally, since $\deg(g_1)=2$, $a_0=1$, and $a_1=0$, we deduce that the moment semidefinite relaxation for the problem \eqref{illustration} will be exact as soon as $n\geq 1+\max\{a_1+1,a_0\}\iff n\geq 2$. 
Indeed, we have,
$-\frac23=\rho_2=f^\star=f(\phi^{(2)}(\x))$ with $\phi^{(2)}(\x)=[-1,0]^\top\in\om$.
\end{ex}
\srecko{\subsection{Discussion}\label{sec: discussion}
We next recall the $Q$-module convexity introduced by Nie
\cite{Nie-2012}. Let $f\in\R[\x]$ and define the first order remainder via:
\begin{eqnarray}
    R_f(\x,\y):=
f(\x)-f(\y)-\nabla f(\y)^\top(\x-\y),\quad \x,\y\in\R^d.
\end{eqnarray}
Then, $f$ is $Q$-module convex over $\om$ if the following representation holds: 
\begin{eqnarray}\label{qmodule-remainder}
R_f(\x,\y)=\sum_{i,j=0}^{m+1}
g_i(\x)g_j(\y)\sigma_{ij}(\x,\y),
\qquad
\sigma_{ij}\in\Sigma[\x,\y].
\end{eqnarray}
This two-variable remainder certificate is related to, but in general \textit{not}
equivalent to, the one-variable Hessian certificate
\eqref{def-sos-om}. Under additional algebraic assumptions on the
constraints, the Hessian certificate implies $Q$-module convexity \cite[Theorem 2.4]{Nie-2012}. This is
the case, in particular, if $\om$ admits a polyhedral description or a suitable
quadratic SOS-concave description.
\hfill\break\\
We emphasize that, although the notion of $Q$-module convexity has not been used in optimization, the analogue of Theorem~\ref{th1} can still be proved by replacing
Assumption~\ref{ass-1} with a degree-bounded representation of the form
\eqref{qmodule-remainder}. Indeed, let $\bphi$ be feasible for an order high
enough to evaluate every term in that representation. Consider its associated Riesz
functional in the $\x$ variables, and set $\y:=\phi(\x)$
{\color{black} (hence a fixed constant vector) to now consider $\x\mapsto R_f(\x,\y)$ as a polynomial in $\x$. Recall that $\y=\phi(\x)\in\om$. 
So in evaluating $\phi(R_f(\cdot,\y))$, feasibility of
$\y=\phi(\x)$ and positivity of the localizing matrices would then give
\begin{eqnarray}
    \phi(f)-f(\phi(\x))\,=\,\phi(f)-f(\y)
=\sum_{i,j=0}^{m+1}\underbrace{g_j(\y)}_{\geq0}\,
\,\underbrace{\phi\bigl(g_i\sigma_{ij}(\cdot,\y)\bigr)}_{\geq0}
\geq0.
\end{eqnarray}
}
The same chain of inequalities as in the proof of Theorem~\ref{th1} yields
finite convergence and recovery of a global minimizer. 
\hfill\break\\
We nevertheless prefer to use \eqref{def-sos-om} {\color{black}(rather than \eqref{qmodule-remainder})} as the main convexity assumption. The reason is that it is naturally verified for every objective that is
strongly convex on $\om$: the matrix Putinar Positivstellensatz supplies the
required Hessian representation. In contrast, $R_f$ can never be strictly positive (when $\x=\y$, we have $R_f(\x,\y)=0$), and therefore, Putinar Positivstellensatz does not apply.
\hfill\break\\
Lastly, it would be interesting to investigate whether analogous finite-convergence guarantees remain valid when $\om$ is convex but this convexity is not apparent from its polynomial description. A natural direction would be to replace global SOS-concavity of the constraint-defining polynomials by Hessian-$Q$-module convexity of each $-g_j$ over $\om$, or another appropriate notion of local convexity. We leave this extension for future work.
}
\section{Conclusion}
\srecko{We have proved finite convergence of the standard Moment--SOS hierarchy for
polynomial optimization problems whose constraints are SOS-concave and whose
objective admits a \textit{hidden} Hessian quadratic-module certificate on the feasible set.
The objective need not be convex on the ambient space. The proof applies to
SOS-concave constraints of arbitrary degree and gives the explicit exactness
order.}
\\
As this hidden convexity is hard to check,
any minimizer on $\om$ (even global) provided by a local algorithm 
is lacking a guarantee  to be a global minimizer. 
Our results complement existing work on convex POPs and highlight a remarkable aspect of the Moment–SOS hierarchy: it can blindly adapt to convex structure in the instance at hand, even when this structure is neither known nor identified in advance. Future work could investigate whether this phenomenon persists for structured variants of the Moment--SOS hierarchy designed to exploit structural properties of considered problems, such as symmetries or different sparsity types.

\end{document}